\documentclass{svmult}

\usepackage{mathptmx}       
\usepackage{helvet}         
\usepackage{courier}        
\usepackage{type1cm}        
\usepackage{makeidx}         
\usepackage{graphicx}        
\usepackage{multicol}        
\usepackage[bottom]{footmisc}
\usepackage{url}
\usepackage{amssymb}
\usepackage{amsmath}
\usepackage{hyperref}

\makeindex             

\begin{document}

\title*{Multi-occupation field generates the Borel-sigma-field of loops.}
\author{Yinshan Chang}
\institute{Yinshan Chang \at Max Planck Institute for Mathematics in the Sciences, 04103 Leipzig, Germany, \email{ychang@mis.mpg.de}
\and This is part of author's PHD work in Department of Mathematics in Universit\'{e} Paris Sud}
%
%
\maketitle

\abstract{In this article, we consider the space of c\`{a}dl\`{a}g loops on a Polish space $S$. The loop space can be equipped with a ``Skorokhod'' metric. Moreover, it is Polish under this metric. Our main result is to prove that the Borel-$\sigma$-field on the space of loops is generated by a class of loop functionals: the multi-occupation field. This result generalizes the result in the discrete case, see \cite{loop}.}

\section{Introduction}
\label{sec:1}
The Markovian loops have been studied by Le Jan \cite{loop} and Sznitman \cite{Sznitman}. Under reasonable assumptions of the state space, as an application of Blackwell's theorem, we would like to prove that multi-occupation field generates the Borel-$\sigma$-field on the space of loops, see Theorem \ref{thm:main theorem}. This generalizes the result in \cite{loop}, see the paragraph below Proposition 10 in Chapter 2 of \cite{loop}. For self-containedness, we introduce several necessary definitions and notations in the following paragraphs.

\bigskip

Let $(S,d_S)$ be a Polish space with the Borel-$\sigma$-field. As usual, denote by $D_S([0,a])$  the Skorokhod space, i.e. the space of c\`{a}dl\`{a}g\footnote{The terminology ``c\`{a}dl\`{a}g'' is short for right-continuous with left hand limits.}-paths over time interval $[0,a]$ which is also left-continuous at time $a$. We equip it with the Skorokhod metric and the corresponding Borel-$\sigma$-field.
\begin{definition}[Based loop]
A based loop is an element $l\in D_S([0,t])$ for some $t>0$ such that $l(0)=l(t)$. We call $t$ the duration of the based loop and denote it by $|l|$.
\end{definition}
\begin{definition}[Loop]
We say two based loops are equivalent iff. they are identical up to some circular translation. A loop is defined as an equivalence class of based loops. For a based loop $l$, we denote by $l^{o}$ its equivalence class.
\end{definition}

\begin{definition}[Multi-occupation field/time]
Define the rotation operator $r_j$ as follows: $r_j(z^1,\ldots,z^n)=(z^{1+j},\ldots,z^n,z^1,\ldots,z^j)$.
For any $f:S^n\rightarrow\mathbb{R}$ measurable, define the multi-occupation field of based loop $l$ of time duration $t$ as $$\langle l,f\rangle=\sum\limits_{j=0}^{n-1}\int\limits_{0<s^1<\cdots<s^n<t}f\circ r_j(l(s^1),\ldots,l(s^n))\,ds^1\cdots\,ds^n.$$
If $l_1$ and $l_2$ are two equivalent based loops, they correspond to the same multi-occupation field. Therefore, the multi-occupation field is well-defined for loops. For discrete $S$, define the multi-occupation time $\hat{l}^{x_1,\ldots,x_n}$ of a (based) loop to be $\langle l,1_{(x_1,\ldots,x_n)}(\cdot)\rangle$ where
$$1_{(x_1,\ldots,x_n)}((y_1,\ldots,y_n))=\left\{\begin{array}{ll}
1 & \text{ if }(y_1,\ldots,y_n)=(x_1,\ldots,x_n)\\
0 & \text{ otherwise.}
\end{array}\right.$$
\end{definition}

The following idea for defining the distance of loops is due to Titus Lupu. Given two based loops $l_1$ and $l_2$, they can be normalized to have duration $1$ by linear time scaling. Denote them $l_1^{\text{normalized}}$ and $l_2^{\text{normalized}}$. As $S$ is Polish, by Theorem 5.6 in \cite{KurtzMR838085}, the Skorokhod space $(D_{S}([0,1]),d)$ is also Polish under the following metric:
\begin{equation}\label{eq:intro.1}
 d(l_1,l_2)\overset{\text{def}}{=}\inf\limits_{\lambda}\left(\sup\limits_{s<t}\left|\log\frac{\lambda(t)-\lambda(s)}{t-s}\right|+\sup\limits_{u\in[0,1]}d_S(l_1(\lambda(u)),l_2(u))\right)
\end{equation}
where the infimum is taken over all increasing bijections $\lambda:[0,1]\rightarrow [0,1]$. Then, it is straightforward to see that the space of based loops under the following metric $D$ is also Polish:
$$D(l_1,l_2)\overset{\text{def}}{=}\Big||l_1|-|l_2|\Big|+d(l_1^{\text{normalized}},l_2^{\text{normalized}}).$$
\begin{definition}[Distance on loops]
Define the distance $D^{o}$ of two loops $l_1^{o}$ and $l_2^{o}$ by
$$D^{o}(l_1^{o},l_2^{o})\overset{\text{def}}{=}\inf\{D(l,l'):l\in l_1^{o}\text{ and }l'\in l_2^{o}\}.$$
\end{definition}
\begin{remark}
 This is not the standard way to define a pseudo metric on quotient space. In general, the above definition does not satisfy the triangular inequality. In this special situation, the distance $D$ is in fact invariant under suitable circular translation which guarantees the triangular inequality.
\end{remark}

We provide the proofs of the following three propositions in Section Appendix.
\begin{proposition}\label{well defined}
 The above distance $D^o$ is well-defined.
\end{proposition}

\begin{proposition}\label{Polish}
The loop space is Polish under the metric $D^o$.
\end{proposition}

Then, we equip the loop space with the Borel-$\sigma$-field. The next proposition states the measurability of the multi-occupation field.

\begin{proposition}\label{measurable}
 Fix any bounded Borel measurable function $f$ on $S^n$, the following map is Borel measurable functional on the loop space:
$$l\rightarrow\langle l,f\rangle.$$
\end{proposition}

Our main result is the following theorem. 

\begin{theorem}\label{thm:main theorem}
The Borel-$\sigma$-field on the loops is generated by the multi-occupation field if $(S,d_S)$ is Polish.
\end{theorem}

\section{Proof of Theorem \ref{thm:main theorem}}
\label{sec:2}
We will prove the main theorem in this section as an application of the following Blackwell's theorem.

\begin{theorem}[Blackwell's theorem, Theorem 26, Chapter III of \cite{meyer}]
Suppose $(E,\mathcal{E})$ is a Blackwell space, $\mathcal{S},\mathcal{F}$ are sub-$\sigma$-field of $\mathcal{E}$ and $\mathcal{S}$ is separable. Then $\mathcal{F}\subset\mathcal{S}$ iff every atom of $\mathcal{F}$ is a union of atoms of $\mathcal{S}$.
\end{theorem}

As a consequence, we have the following lemma.

\begin{lemma}\label{lem:a corollary of the theorem of Blackwell}
Suppose $(E,\mathcal{B}(E))$ is a Polish space with the Borel-$\sigma$-field. Let $\{f_i,i\in\mathbb{N}\}$ be measurable functions and denote $\mathcal{F}=\sigma(f_i,i\in\mathbb{N})$. Then, $\mathcal{F}=\mathcal{B}(E)$ iff for all $x\neq y\in E$, there exists $f_i$ such that $f_i(x)\neq f_i(y)$.
\end{lemma}
\begin{proof}
Since $E$ is Polish, $\mathcal{B}(E)$ is separable and $(E,\mathcal{B}(E))$ is Blackwell space. The atoms of $\mathcal{B}(E)$ are all the one point sets. Obviously, $\mathcal{F}\subset\mathcal{B}(E)$ and $\mathcal{F}$ is separable. By Blackwell's theorem, $\mathcal{F}=\mathcal{B}(E)$ iff. the atoms of $\mathcal{F}$ are all the one point sets which is equivalent to the following: for all $x\neq y\in E$, there exists $f_i$ such that $f_i(x)\neq f_i(y)$.
\end{proof}

Then, we are ready for the proof of the main theorem.

\bigskip

From the definition of the multi-occupation field, any loop defines a finite measure on $S^n$ for all $n\in\mathbb{N}_{+}$. Let $\mathfrak{B}=(B_i,i\in\mathbb{N})$ be a countable topological basis of $S$. The $\sigma-$field generated by the multi-occupation field must equal to the $\sigma-$field generated by the following countable\footnote{The countability is required by Lemma \ref{lem:a corollary of the theorem of Blackwell}.} functionals $\{\langle \cdot,1_{B}\rangle:B\in\bigcup\limits_{k=1}^{\infty}\mathfrak{B}^k\}$. In fact, if two loop $l_1$ and $l_2$ are the same under these countable loop functionals $\{\langle \cdot,1_{B}\rangle:B\in\bigcup\limits_{k=1}^{\infty}\mathfrak{B}^k\}$, they must agree on all the functionals of the form $\langle \cdot,f\rangle$. By Lemma \ref{lem:a corollary of the theorem of Blackwell}, it remains to check that two loops with the same occupation field are the same loop.

\bigskip

Suppose loops $l_1^o$ and $l_2^o$ have the same occupation field, i.e. $\langle l_1^o,f\rangle=\langle l_2^o,f\rangle$ for all positive $f$ on some $S^n$ $(n\in\mathbb{N}_{+})$. Recall that a loop is an equivalence class of based loops. Take two based loops $l_1,l_2$ in the equivalence class $l_1^o,l_2^o$ respectively. Then, $\langle l_1,f\rangle=\langle l_1^o,f\rangle=\langle l_2^o,f\rangle=\langle l_2,f\rangle$. Define $m_1(A)=\langle l_1,1_{A}\rangle$ and $m_2(A)=\langle l_2,1_{A}\rangle$ for $A\in\mathcal{B}(S)$. Then, we have $m_1=m_2$ which means that the time spent in some Borel measurable set is the same for these two loops. In particular, the two (based) loops have the same time duration, say $t$. For simplicity of the notations, we will use $m$ instead of $m_1$ and $m_2$. Now, we are ready to show that $l_1^{o}=l_2^{o}$ in three steps. Let us present the sketch of the proof before providing the details. We first decompose the space into an approximate partition which is used in \cite{QianMR3101845}. Next, we replace arcs of trajectory in each part by a single point with corresponding holding times. In this way, we get two loops in the same discrete space. By the construction of these discrete loops, their multi-occupation fields coincide. It is known that Theorem \ref{thm:main theorem} is true for loops in discrete space. Thus, these two modified loops in the discrete space are exactly the same. Moreover, when the rough partition is small enough, these modified loops are actually good approximation of the original loops $l_1^{o}$ and $l_2^{o}$ in the sense of Skorokhod. For that reason, we conclude in the last step that $l_1^{o}=l_2^{o}$.

\bigskip

\begin{itemize}
\item[I.]For all $\epsilon>0$ fixed, we choose a collection of open sets $U^{\epsilon}_i$ satisfying the following properties:
\begin{itemize}
\item their boundaries are negligible with respect to $m$,
\item they have positive distances from each other,
\item the complement of their union has mass smaller than $2\epsilon$ with respect to $m$,
\item their diameters are smaller than $\epsilon$.
\end{itemize}
Let $U^{\epsilon}$ be the union of $(U^{\epsilon}_i)_i$.\\


Actually, the rough partitions $(U^{\epsilon}_i)_i$ are chosen in the following way. It is well-known that every finite measure on the Borel-$\sigma$-field of a Polish space is regular. Therefore, for $\epsilon>0$, we can find some compact set $K_{\epsilon}$ such that $m(K_{\epsilon}^c)<\epsilon$ where $K_{\epsilon}^c$ is the complement of $K_{\epsilon}$. Let $\mathcal{D}=\{x_1,\cdots,x_n,\cdots\}$ be a countable dense subset of $S$. Fix any $x\in S$, except for countable many $r\in\mathbb{R}_{+}$, the measure $m$ does not charge the boundary $\partial(B(x,r))$ of the ball $B(x,r)$. Then, for any $\epsilon>0$, there exists a collection of open ball $B(x_i,r_i)$ with radius smaller than $\epsilon$ such that their boundaries are negligible with respect to $m$. Then, they cover the compact set $K_{\epsilon}$ as $\mathcal{D}$ is dense in $S$. Therefore, we can extract a finite open covering $\{B(y_1,r_1),\cdots,B(y_{k},r_k)\}$. These open balls cut the whole space $S$ into a finite partition of the space $S\setminus\bigcup\limits_{i}\partial(B(y_i,r_i))$: $P_0,\cdots,P_{q}$ open set with $P_0=(\bigcup\limits_{i} \overline{B(y_i,r_i)})^c$. Let $U^{\epsilon}_{i,\delta}=\{y\in S:d_S(y,P_i^c)>\delta\}$ which is contained in $P_i$. In fact, one can always choose some $\delta_0$ small enough and good enough such that the boundary sets $\{y\in S:d_S(y,P_i^c)=\delta_0\}$ of these open sets are negligible under $m$. Moreover, $m(S\setminus(\bigcup\limits_{i}U^{\epsilon}_{i,\delta_0}\cap K_{\epsilon}))<2\epsilon$. Set $U^{\epsilon}_{i}=U^{\epsilon}_{i,\delta_0},i=1,\cdots,q$. Then, they satisfy the desired properties stated above.

\item[II.]From the based loop $l_j$ $(j=1,2)$, we will construct two piecewise-constant based loops $l^{\epsilon}_{j}$ $(j=1,2)$ with finitely many jumps such that $l^{\epsilon}_{1}$ and $l^{\epsilon}_{2}$ are the same in the sense of loop and that they are quite close to the trace of $l_1$ and $l_2$ on $U^{\epsilon}$ respectively.\\

To be more precise, define $A^{\epsilon}_{j,u}=\int\limits_{0}^{u}1_{\{l_j(s)\in U^{\epsilon}\}}\,ds$ with the convention that $l_j(s+kt)=l_j(s)$ for $s\in[0,t]$ and $k\in\mathbb{Z}$ where $t$ is the time duration of the based loops. Then, $(A^{\epsilon}_{j,u},u\in\mathbb{R}_{+})$ is right-continuous and increasing for $j=1,2$. Let $(\sigma^{\epsilon}_{j,s},s\in\mathbb{R}_{+})$ be the right-continuous inverse of $(A^{\epsilon}_{j,u},u\in\mathbb{R}_{+})$ for $j=1,2$ respectively. To be more precise, $$\sigma^{\epsilon}_{j,s}=\inf\{s\in\mathbb{R}_{+}:A^{\epsilon}_{j,u}>s\}.$$
Let $t_{\epsilon}=A^{\epsilon}_{1,t}=A^{\epsilon}_{2,t}=m(U^{\epsilon})$ to be the total occupation time of the loops within $U^{\epsilon}$. Then, $A^{\epsilon}_{j,u+kt}=kt_{\epsilon}+A^{\epsilon}_{j,u}$ and $A^{\epsilon}_{j,u}\leq u$ for $u\in\mathbb{R}_{+}$. Thus, $\sigma^{\epsilon}_{j,s+kt_{\epsilon}}=\sigma^{\epsilon}_{j,s}+kt$ for $k\in\mathbb{N},s\in[0,t_{\epsilon}[$ and $\sigma^{\epsilon}_{j,s}\geq s$ for $s\in\mathbb{R}_{+}$. Moreover, as $\epsilon\downarrow 0$, $(\sigma^{\epsilon}_{j,s},s\in\mathbb{R}_{+})$ decreases to $(s,s\in\mathbb{R}_{+})$ uniformly on any compact of $\mathbb{R}_{+}$. We know that $l_j(\sigma^{\epsilon}_{j,s})\in\bigcup\limits_{i}\overline{U^{\epsilon}_i}$. We choose in each $U^{\epsilon}_i$ a point $y_i$ and define $l^{\epsilon}_{j}(s)=y_i$ iff. $l_j(\sigma^{\epsilon}_{j,s})\in \overline{U^{\epsilon}_i}$ for $j=1,2$. Then, as the diameters of $(U^{\epsilon}_i)_i$ are less than $\epsilon$, $\sup\limits_{s}d_S(l^{\epsilon}_{j}(s),l_j(\sigma^{\epsilon}_{j,s}))\leq\epsilon$ for $j=1,2$. Moreover, $s\rightarrow l^{\epsilon}_{j}(s)$ is c\`{a}dl\`{a}g for $j=1,2$. Since the based loops $l_1$ and $l_2$ are c\`{a}dl\`{a}g and all the $\overline{U^{\epsilon}_{i}}$ have a positive distance from each other, $s\rightarrow l^{\epsilon}_{j}(s)$ has finitely many jumps in any finite time interval for $j=1,2$. Then, $(l^{\epsilon}_{j}(s),s\in[0,t_{\epsilon}])^{o}$ is a loop on the same finite state space for $j=1,2$ respectively. Since the boundary of $U^{\epsilon}_i$ in negligible with respect to $m$, by Lebesgue's change of measure formula,
$$((l^{\epsilon}_{j}(s),s\in[0,t_{\epsilon}])^{o})^{y_{i_1},\cdots,y_{i_n}}=\langle l_j,1_{U^{\epsilon}_{i_1}}\cdots 1_{U^{\epsilon}_{i_n}} \rangle\text{ for }j=1,2.$$
Therefore, $(l^{\epsilon}_{1}(s),s\in[0,t_{\epsilon}])^{o}$ and $(l^{\epsilon}_{2}(s),s\in[0,t_{\epsilon}])^{o}$ have the same multi-occupation field. It is known that Theorem \ref{thm:main theorem} is true for loops in finite discrete space, see the paragraph below Proposition 10 in Chapter 2 of \cite{loop}. Thus, 
$$(l^{\epsilon}_{1}(s),s\in[0,t_{\epsilon}])^{o}=(l^{\epsilon}_{2}(s),s\in[0,t_{\epsilon}])^{o}.$$
Consequently, there exists some $T_2(\epsilon)\in[0,t_{\epsilon}[$ such that $l^{\epsilon}_{2}(s+T_2)=l^{\epsilon}_{1}(s)$ for $s\geq 0$.

\item[III.] Take the limit on a subsequence and use the right-continuity of the path to conclude $l_1=l_2$ up to circular translation.\\
We can find a sequence $(\epsilon_k)_{k}$ with $\lim\limits_{k\rightarrow\infty}\epsilon_k=0$ such that $T_2(\epsilon_k)$ converges to $T\in[0,t]$ as $k\rightarrow\infty$. Then, $\lim\limits_{k\rightarrow\infty}\sigma^{\epsilon_k}_{2,s+T_2}=s+T$ for fixed $s\geq 0$. Accordingly, 
\begin{equation}\label{eq:2.3.1}
\lim\limits_{k\rightarrow\infty}\min\{d_S(l_2(\sigma^{\epsilon_k}_{2,s+T_2}),l_2(s+T)),d_S(l_2(\sigma^{\epsilon_k}_{2,s+T_2}),l_2((s+T)-))\}=0.\end{equation}
On the other hand, we have $\lim\limits_{k\rightarrow\infty}\sigma^{\epsilon_k}_{2,s}=s$ and $\sigma^{\epsilon_k}_{2,s}\geq s$ for all $k\in\mathbb{N}$. Therefore, by the right continuity of $l_1$,
\begin{equation}\label{eq:2.3.2}
\lim\limits_{k\rightarrow\infty}d_S(l_1(\sigma^{\epsilon_k}_{1,s}),l_1(s))=d_S(l_1(s+),l_1(s))=0.\end{equation}
From the constructions of $l_1^{\epsilon}$ and $l_2^{\epsilon}$ and the argument in part II, we see that
\begin{align}\label{eq:2.3.3}
\sup\limits_{s}d_S(l_2(\sigma^{\epsilon_k}_{2,s+T_2}),l_1(\sigma^{\epsilon_k}_{1,s}))\leq & \sup\limits_{s}d_S(l^{\epsilon}_2(s+T_2),l_2(\sigma^{\epsilon_k}_{2,s+T_2}))\nonumber\\
&+\sup\limits_{s}d_S(l^{\epsilon}_1(s),l_1(\sigma^{\epsilon_k}_{1,s}))\nonumber\\
\leq &2\epsilon_k.
\end{align}
As a result, by \eqref{eq:2.3.1}$+$\eqref{eq:2.3.2}$+$\eqref{eq:2.3.3}, for any $s\geq 0$, either $d_S(l_2(s+T),l_1(s))=0$ or $d_S(l_2((s+T)-),l_1(s))=0$. Finally, by right-continuity of the paths $l_1$ and $l_2$,
$$l_2(s+T)=l_1(s)$$ and the proof is complete.
\end{itemize}

\section*{Appendix}
\addcontentsline{toc}{section}{Appendix}
As promised, we give the proofs for Proposition \ref{well defined}, \ref{Polish} and \ref{measurable} in this section. For that reason, we prepare several notations and lemmas in the following.

\bigskip

\begin{definition}
Suppose $\lambda:[0,1]\rightarrow [0,1]$ is a increasing bijection. For $t\in[0,1[$, define $$\theta_t\lambda(s)=\left\{
\begin{array}{ll}
 \lambda(t+s)-\lambda(t) & \text{ for } s\in[0,1-t]\\
 1-\lambda(t)+\lambda(t+s-1) & \text{ for }s\in[1-t,1].
\end{array}
\right.$$
\end{definition}
In fact, we cut the graph of $\lambda$ at the time $t$, exchange the first part of the graph with the second part and then glue them together to get an increasing bijection over $[0,1]$.
\begin{lemma}\label{lem:invariant under regathering}
 $$\sup\limits_{s<t}\left|\log\frac{\theta_r\lambda(t)-\theta_r\lambda(s)}{t-s}\right|=\sup\limits_{s<t}\left|\log\frac{\lambda(t)-\lambda(s)}{t-s}\right|.$$
\end{lemma}
\begin{proof}
 Denote by $\phi(\lambda,s,t)$ the quantity $|\log\frac{\lambda(t)-\lambda(s)}{t-s}|$. We see that
$$\max(\phi(\lambda,a,b),\phi(\lambda,b,c))\geq \phi(\lambda,a,c).$$
Thus, $\sup\limits_{s<t}\phi(\lambda,s,t)=\sup\limits_{s<t,t-s\text{ is small}}\phi(\lambda,s,t)$. As a result, $\sup\limits_{s<t}|\log\frac{\lambda(t)-\lambda(s)}{t-s}|$ is a function of $\lambda$ which is invariant under $\theta_t$.
\end{proof}
\begin{definition}
For a based loop $l$ of time duration $t$ and $r\in[0,t[$, denote by $\Theta_r$ the circular translation of $l$:
$$\Theta_r(l)(u)=\left\{
\begin{array}{ll}
l(u+r) & \text{ for }u\in[0,t-r]\\
l(u+r-t) & \text{ for }u\in[t-r,t].                         
\end{array}\right.$$
Then, we can extend $\Theta_r$ for all $r\in\mathbb{R}$ by periodical extension.
\end{definition}
Notice that $\Theta_r(l)$ is a based loop iff. the periodical extension of $l$ is continuous at time $r$. Nevertheless, we define the distance $D(\Theta_r l,l)$ in the same way. The next lemma shows the continuity of $r\rightarrow \Theta_rl$ at time $r$ when the based loop $l$ is continuous at $r$. 
\begin{lemma}\label{lem:continuous under circular translation}
 Suppose $l$ is a based loop. Then, $\lim\limits_{h\rightarrow 0}D(\Theta_h l,l)=0$.
\end{lemma}
\begin{proof}
Without loss of generality, we can assume $l$ has time duration $1$. By definition, we have that
\begin{align*}
 D(\Theta_h(l),l)=&d(\Theta_h(l),l)\\
 =&\inf\Big\{\sup\limits_{s<t}\left|\log\frac{\lambda(t)-\lambda(s)}{t-s}\right|+\sup\limits_{u\in [0,1]}d_S\left(l(\lambda(u)),\Theta_h(l)\right):\\
&\lambda\text{ increasing bijection on }[0,1]\Big\}.
\end{align*}
Fix $0<a<b<1$, take $\lambda(0)=0,\lambda(a)=a+h,\lambda(b)=b+h,\lambda(1)=1$ and linearly interpolate $\lambda$ elsewhere. Then, 
\begin{align*}
 D(\Theta_h(l),l)\leq &\max\left(\left|\log\frac{a+h}{a}\right|,\left|\log\frac{1-b-h}{1-b}\right|\right)\\
&+2\sup\limits_{u,v\in[0,a+|h|]\cup[b-|h|,1]}|l(u)-l(v)|.
\end{align*}
Thus, for any $0<a<b<1$,
$$\limsup\limits_{h\rightarrow 0}D(\Theta_h(l),l)\leq 2\sup\limits_{u,v\in[0,a]\cup[b,1]}|l(u)-l(v)|.$$
Since $l$ is a based loop, $\inf\limits_{a,b}(\sup\limits_{u,v\in[0,a]\cup[b,1]}|l(u)-l(v)|)=0$. Therefore,
$$\lim\limits_{h\rightarrow 0}D(\Theta_h l,l)=0.$$
\end{proof}

\begin{lemma}\label{lem:invariant under circular translation}
Suppose $l_1$ is a based loop with time duration $t$ and $l$ is continuous at time $r\in[0,t[$. Then,
 $$\inf\{D(l_1,l):l\in l_2^{o}\}=\inf\{D(\Theta_r(l_1),l):l\in l_2^{o}\}.$$
\end{lemma}
\begin{proof}
Recall that $D(l_1,l)=\Big||l|-|l_1|\Big|+d(l_1^{\text{normalized}},l^{\text{normalized}})$ where 
\begin{multline*}
d(l_1^{\text{normalized}},l^{\text{normalized}})=\inf\Big\{\sup\limits_{u\in[0,1]}d_S(l_1^{\text{normalized}}(u),l^{\text{normalized}}(\lambda(u)))\\
+\sup\limits_{s<t}\left|\log\frac{\lambda(t)-\lambda(s)}{t-s}\right|:\lambda\text{ increasing bijection over }[0,1]\Big\}.
\end{multline*}
Then, for $\epsilon>0$, there exists $l\in l_2^o$ and $\lambda$ such that
\begin{multline}\label{eq:apd.1}
\sup\limits_{s<t}\left|\log\frac{\lambda(t)-\lambda(s)}{t-s}\right|+\sup\limits_{u\in[0,1]}d_S(l_1^{\text{normalized}}(u),l^{\text{normalized}}(\lambda(u)))\\
<\inf\{D(l_1,l):l\in l_2^{o}\}+\epsilon.
\end{multline}

Since the paths are c\`{a}dl\`{a}g, the following set is at most countable:
$$\{a:l_1\text{ jumps at time }a\text{ or }l\text{ jumps at }|l|\lambda(a/|l_1|)\}.$$
Thus, we can find a sequence $(r_n)_n$ such that
\begin{itemize}
 \item $r_n\downarrow r$ as $n\rightarrow \infty$,
 \item $\Theta_{r_n}(l_1)$ and $\Theta_{|l|\lambda(r_n/|l_1|)}(l)$ are both based loops.
\end{itemize}
By Lemma \ref{lem:invariant under regathering}, we have that
\begin{equation}\label{eq:apd.2}
 \sup\limits_{s<t}\left|\log\frac{\lambda(t)-\lambda(s)}{t-s}\right|=\sup\limits_{s<t}\left|\log\frac{\theta_{r_n/|l_1|}\lambda(t)-\theta_{r_n/|l_1|}\lambda(s)}{t-s}\right|.
\end{equation}
Meanwhile, we have that
\begin{multline}\label{eq:apd.3}
\sup\limits_{u\in[0,1]}d_S(l_1^{\text{normalized}}(u),l^{\text{normalized}}(\lambda(u)))\\
=\sup\limits_{u\in[0,1]}d_S\left((\Theta_{r_n}l_1)^{\text{normalized}}(u),(\Theta_{|l|\lambda(r_n/|l_1|)}l)^{\text{normalized}}(\theta_{r_n/|l_1|}\lambda(u))\right).
\end{multline}
Notice that $\Theta_{|l|\lambda(r_n/|l_1|)}l\in l_2^{o}$. Thus, by \eqref{eq:apd.1}+\eqref{eq:apd.2}+\eqref{eq:apd.3}, for any $\epsilon>0$, there exists $(r_n)_n$ with decreasing limit $r$ such that
\begin{equation}\label{eq:apd.4}
\inf\{D(\Theta_{r_n}l_1,l):l\in l_2^{o}\}<\inf\{D(l_1,l):l\in l_2^{o}\}+\epsilon.
\end{equation}
By triangular inequality of $D$,
$$D(\Theta_{r}l_1,l)\leq D(\Theta_{r_n}l_1,\Theta_{r}l_1)+D(\Theta_{r_n}l_1,l).$$
We take the infimum on both sides, then
\begin{equation*}
\inf\{D(\Theta_{r}l_1,l):l\in l_2^o\}\leq D(\Theta_{r_n}l_1,\Theta_{r}l_1)+\inf\{D(\Theta_{r_n}l_1,l):l\in l_2^o\}.
\end{equation*}
By \eqref{eq:apd.4},
\begin{equation}\label{eq:apd.5}
\inf\{D(\Theta_{r}l_1,l):l\in l_2^o\}\leq D(\Theta_{r_n}l_1,\Theta_{r}l_1)+\inf\{D(l_1,l):l\in l_2^{o}\}+\epsilon.
\end{equation}
By Lemma \ref{lem:continuous under circular translation}, for based loop $l$, $\lim\limits_{n\rightarrow\infty}D(\Theta_{r_n}l_1,\Theta_{r}l_1)=0$. By taking $n\rightarrow\infty$ in \eqref{eq:apd.5}, we see that
$$\inf\{D(\Theta_{r}l_1,l):l\in l_2^o\}\leq\inf\{D(l_1,l):l\in l_2^o\}+\epsilon\text{ for all }\epsilon>0.$$
Therefore,
$$\inf\{D(\Theta_{r}l_1,l):l\in l_2^o\}\leq\inf\{D(l_1,l):l\in l_2^o\}.$$
If we replace $l_1$ by $\Theta_{r}l_1$ and $r$ by $|l_1|-r$, we have the inequality in opposite direction:
$$\inf\{D(\Theta_{r}l_1,l):l\in l_2^o\}\geq\inf\{D(l_1,l):l\in l_2^o\}.$$
\end{proof}

Then, we turn to prove Proposition \ref{well defined}, \ref{Polish} and \ref{measurable}.

\begin{proof}[Proof of Proposition \ref{well defined}]\ 
\begin{itemize}
 \item Reflexivity: straightforward from the definition. 
 \item Triangular inequality: directly from Lemma \ref{lem:invariant under circular translation}.
 \item $D^{o}(l_1^{o},l_2^{o})=0\Longrightarrow l_1^o=l_2^o$: by Lemma \ref{lem:invariant under circular translation}, it is enough to show that
$$\inf\{D(l_1,l):l\in l_2^o\}=0\Longrightarrow l_1\in l_2^{o}.$$
Suppose $\inf\{D(l_1,l):l\in l_2^o\}=0$. Then, we can find a sequence $(r_n)_n$ with limit $r$ such that $\lim\limits_{n\rightarrow\infty}D(\Theta_{r_n}l_2,l_1)=0$. Since $l_1(|l_1|-)=l_1(0)$, $l_2$ must be continuous at $r$ and $\lim\limits_{n\rightarrow\infty}\Theta_{r_n}l_2=\Theta_rl_2$ by Lemma \ref{lem:continuous under circular translation}. Thus, $l_1=\Theta_rl_2$.
\end{itemize}
\end{proof}

\begin{proof}[Proof of Proposition \ref{Polish}]\ 
\begin{itemize}
 \item Completeness: given a Cauchy sequence $(l_n^{o})_n$, one can always extract a sub-sequence $(l_{n_k}^{o})_k$ such that $D^{o}(l_{n_k}^{o},l_{n_{k+1}}^{o})<2^{-k}$. By Lemma \ref{lem:invariant under circular translation}, one can find in each equivalence class $l_{n_k}^{o}$ a based loop $L_k$ such that $D(L_{k},L_{k+1})<2^{-k}$. By the completeness of $D$, there exists a based loop $L$ such that $\lim\limits_{k\rightarrow\infty}L_k=L$. Thus, $\lim\limits_{k\rightarrow\infty}l_{n_k}^{o}=L^{o}$. So is the same for $(l_n^{o})_n$.
 \item Separability: the based loop space is separable. Then, as the continuous image, the loop space is separable.
\end{itemize}
\end{proof}

\begin{proof}[Proof of Proposition \ref{measurable}]
For bounded continuous function $f:S^n\rightarrow \mathbb{R}$, $l\rightarrow\langle l,f\rangle$ is continuous in $l$. In particular, it is measurable. By $\pi-\lambda$ theorem for functions, $l\rightarrow\langle l,f\rangle$ is measurable for all bounded measurable $f:S^n\rightarrow\mathbb{R}$.
\end{proof}

\bibliographystyle{amsalpha}
\bibliography{reference}
\end{document}